\documentclass[10pt,reqno,a4paper]{amsart}
\usepackage{amsfonts,latexsym,amssymb,amsmath,amsthm,mathtools}
\usepackage{amsrefs} 
\usepackage{enumerate}
\usepackage[all]{xy}
\usepackage{todonotes}
\usepackage{hyperref}
\usepackage{newtxtext,newtxmath}

\makeatletter
\g@addto@macro\bfseries{\boldmath}
\makeatother

\hypersetup{
    colorlinks,
    linkcolor={red!50!black},
    citecolor={blue!50!black},
    urlcolor={blue!80!black}
}

\newcommand{\nE}{\nabla^E}

\newcommand{\dA}{\mathrm{dA}}

\newcommand{\Z}{\mathbb{Z}}
\newcommand{\R}{\mathbb{R}}

\newcommand{\id}{\mathrm{id}}
\newcommand{\ep}{\varepsilon}
\newcommand{\rmi}{\mathrm{i}}
\newcommand{\APS}{\mathrm{APS}}

\newcommand{\FE}{{F\hspace{-0.8mm}E}}

\newcommand{\res}{\mathrm{res}}

\newcommand{\Adach}{{\widehat{\mathrm{A}}}}

\newcommand{\ch}{\mathrm{ch}}
\newcommand{\eps}{\varepsilon}
\newcommand{\gt}{\tilde g}

\newcommand{\<}{\langle}
\renewcommand{\>}{\rangle}

\newcommand{\ind}{\mathrm{ind}}
\newcommand{\adj}{^\dagger}
\newcommand{\coker}{\mathrm{coker}}

 \newtheorem{theorem}{Theorem}[section]
 \newtheorem{lem}[theorem]{Lemma}
 \newtheorem{cor}[theorem]{Corollary}
 \newtheorem{proposition}[theorem]{Proposition}

 \theoremstyle{definition}
 \newtheorem{definition}[theorem]{Definition}
 \newtheorem{example}[theorem]{Example}
 \newtheorem{remark}[theorem]{Remark}

\allowdisplaybreaks
\title[Boundary value problems for the Lorentzian Dirac operator]{Boundary value problems for the Lorentzian Dirac operator}

\author[C.~B\"ar]{Christian B\"ar}
\author[S.~Hannes]{Sebastian Hannes}
\address{Institut f\"ur Mathematik, Universit\"at Potsdam, Karl-Liebknecht-Str.~24-25, 14476 Potsdam, Germany}
\email{baer@math.uni-potsdam.de, shannes@math.uni-potsdam.de}
\dedicatory{Dedicated to Nigel Hitchin on the occasion of his 70th birthday}

\date{\today}

\begin{document}

\begin{abstract}
On a compact globally hyperbolic Lorentzian spin manifold with smooth spacelike Cauchy boundary the (hyperbolic) Dirac operator is known to be Fredholm when Atiyah-Patodi-Singer boundary conditions are imposed.
In this paper we investigate to what extent these boundary conditions can be replaced by more general ones and how the index then changes.
There are some differences to the classical case of the elliptic Dirac operator on a Riemannian manifold with boundary.
\end{abstract}
\maketitle


\section{Introduction}

The Atiyah-Singer index theorem \cite{AS63} for elliptic operators on closed manifolds is one of the central mathematical discoveries of the 20$^\mathrm{th}$ century.
It contains famous classical results such as the Gauss-Bonnet theorem, the Riemann-Roch theorem or Hirzebruch's signature theorem as special cases and has numerous applications in analysis, geometry, topology, and mathematical physics.
For instance, it has been used in \cite{L63} to obtain a topological obstruction to the existence of metrics with positive scalar curvature and a refinement of the index theorem was employed in \cite{H74} to show that on many manifolds a change of metric in a neighborhood of a point will alter the dimension of the space of harmonic spinors.
This contrasts with the space of harmonic forms whose dimensions are given topologically by the Betti numbers.

The index theorem for compact manifolds with boundary by Atiyah, Patodi, and Singer \cite{APS75} requires the introduction of suitable nonlocal boundary conditions which are based on the spectral decomposition of the operator induced on the boundary.
An exposition of the most general boundary conditions which one can impose in order to obtain a Fredholm operator can be found e.g.\ in \cite{BB16}.

While an analog of the Atiyah-Singer index theorem for Lorentzian manifolds is unknown and not to be expected, one for Lorentzian manifolds with spacelike boundary has been found recently \cite{BS15}.
More precisely, we consider the (twisted) Dirac operator on a spatially compact globally hyperbolic manifold.
It is supposed to have boundary consisting of two disjoint smooth spacelike Cauchy hypersurfaces.
The Dirac operator is now hyperbolic rather than elliptic but the operator induced on the boundary is still selfadjoint and elliptic so that Atiyah-Patodi-Singer boundary conditions still make sense.
It was shown in \cite{BS15} that under these boundary conditions the Dirac operator becomes Fredholm and its index is given formally by the same geometric expression as in the Riemannian case.
As an application the chiral anomaly in algebraic quantum field theory on curved spacetimes was computed in \cite{BS16}.

In the present paper we investigate more general boundary conditions which turn the hyperbolic Dirac operator into a Fredholm operator.
There are similarities and differences to the Riemannian case.
It was already observed in \cite{BS16} that the boundary conditions complementary to the APS boundary conditions, the anti-Atiyah-Patodi-Singer boundary conditions, also give rise to a Fredholm operator.
This is false in the Riemannian case.
On the other hand, we show that the conditions described in \cite{BB12,BB16} for the Riemannian case work in the Lorentzian setting only under an additional assumption.
One can think of these boundary conditions as graph deformations of the APS boundary conditions plus finite dimensional modifications.
The finite dimensional modifications work just the same in the Lorentzian setting but the deformations need to be small, either in the sense that the linear maps whose graphs we are considering are compact or that they have sufficiently small norm.
We show by example that these conditions cannot be dropped.

The paper is organized as follows.
In the next section we summarize what we need to know about Dirac operators on Lorentzian manifolds.
The most important fact is well-posedness of the Cauchy problem.
In the third section we discuss some functional-analytic topics concerning Fredholm pairs.
In the last section we combine everything and consider various examples of boundary conditions giving rise to Fredholm operators and compute their index.

\section{The Dirac operator on Lorentzian manifolds}
\label{sec:setup}

We collect a few standard facts on Dirac operators on Lorentzian manifolds.
For a more detailed introduction to Lorentzian geometry see e.g.\ \cite{BEE96,ON83}, for Dirac operators on semi-Riemannian manifolds see \cite{BGM05,B81}.

\subsection{Globally hyperbolic manifolds}
Suppose that $M$ is an $(n+1)$-dimensional oriented time-oriented Lorentzian spin manifold with $n$ odd. 
We use the convention that the metric of $M$ has signature $(-+\cdots +)$.

A subset $\Sigma\subset M$ is called a \emph{Cauchy hypersurface} if every inextensible timelike curve in $M$ meets $\Sigma$ exactly once.
If $M$ possesses a Cauchy hypersurface then $M$ is called \emph{globally hyperbolic}.
All Cauchy hypersurfaces of $M$ are homeomorphic.
We assume that $M$ is \emph{spatially compact}, i.e.\ the Cauchy hypersurfaces of $M$ are compact.

If $\Sigma_0,\Sigma_1\subset M$ are two disjoint smooth and spacelike Cauchy hypersurfaces with $\Sigma_0$ lying in the past of $\Sigma_1$ then $M$ can be written as 
\begin{equation}
M = \R \times \Sigma
\label{eq:Split}
\end{equation}
such that $\Sigma_0= \{ 0 \} \times \Sigma$, $\Sigma_1= \{ 1 \} \times \Sigma$ and each $\Sigma_t = \{ t \} \times \Sigma$ is a smooth spacelike Cauchy hypersurface.
The metric of $M$ then takes the form $\<\cdot,\cdot\> = -N^2\, dt^2 + g_t$ where $N:M\to\R$ is a smooth positive function (the lapse function) and $g_t$ is a smooth $1$-parameter family of Riemannian metrics on $\Sigma$, see \cite[Thm.~1.2]{BS06} (and also \cite[Thm.~1]{M16}).

\subsection{Spinors}
Let $SM \to M$ be the complex spinor bundle on $M$ endowed with its invariantly defined indefinite inner product $(\cdot,\cdot)$.
Denote Clifford multiplication with $X\in T_pM$ by $\gamma(X):S_pM\to S_pM$.
It satisfies
$$
\gamma(X)\gamma(Y) + \gamma(Y)\gamma(X) = - 2 (X,Y)
$$
and 
$$
(\gamma(X) u,v) = (u,\gamma(X) v)
$$
for all $X,Y\in T_pM$, $u,v \in S_pM$ and $p\in M$.

Let $e_0,e_1,\ldots,e_n$ be a positively oriented Lorentz-orthonormal tangent frame.
Then Clifford multiplication with the volume form $\Gamma=\rmi^{n(n+3)/2}\,\gamma(e_0)\cdots\gamma(e_n)$ satisfies $\Gamma^2=\id_{SM}$.
This induces the eigenspace decomposition $SM = S^RM \oplus S^LM$ for the eigenvalues $\pm1$ into right-handed and left-handed spinors.
Since the dimension of $M$ is even, $\Gamma\gamma(X)=-\gamma(X)\Gamma$ for all $X\in TM$.
In particular, $S^RM$ and $S^LM$ have equal rank and Clifford multiplication by tangent vectors reverses handedness.

Now let $\Sigma\subset M$ be a smooth spacelike hypersurface.
Denote by $\nu$ be the past-directed timelike vector field on $M$ along $\Sigma$ with $\<\nu,\nu\>\equiv -1$ which is perpendicular to $\Sigma$.
The restriction of $S^{R} M$ or $S^{L} M$ to $\Sigma$ can be naturally identified with the spinor bundle of $\Sigma$, i.e.\ $S^{R} M|_{\Sigma} = S^{L} M|_{\Sigma} = S\Sigma$.
The spinor bundle of $\Sigma$ carries a natural positive definite scalar product $\<\cdot,\cdot\>$ induced by the Riemannian metric of $\Sigma$.
The two inner products are related by $\langle \cdot, \cdot \rangle=(\gamma(\nu) \cdot, \cdot)$.

Clifford multiplication $\gamma_{\Sigma}(X)$ on $S\Sigma$ corresponds to $\rmi \gamma(\nu) \gamma(X)$ under this identification.
Note that $\gamma(\nu)^2=\id$.
Clifford multiplication on $\Sigma$ is skew-adjoint because
\begin{align*}
\<\gamma_{\Sigma}(X)u, v\> 
&= (\gamma(\nu)\rmi\gamma(\nu)\gamma(X)u, v) 
= (\rmi\gamma(X)u, v) 
= -(u,\rmi\gamma(X) v) \\
&= -(u,\gamma(\nu)\gamma_{\Sigma}(X) v) 
= -(\gamma(\nu) u,\gamma_{\Sigma}(X) v) 
= -\<u,\gamma_{\Sigma}(X) v\> \, .
\end{align*}

\subsection{The Dirac operator}
Let $E\to M$ be a Hermitian vector bundle with a compatible connection $\nE$.
Sections of the vector bundles $V^R:=S^RM\otimes E$ and $V^L:=S^LM\otimes E$ are called right-handed (resp.\ left-handed) twisted spinors (or spinors with coefficients in $E$).
The inner product $(\cdot,\cdot)$ on $S^RM$ and the scalar product on $E$ induce an (indefinite) inner product on $V^R$, again denoted by $(\cdot,\cdot)$.
When restricted to a spacelike hypersurface the scalar product $\<\cdot,\cdot\>$ on $S^RM$ and the one on $E$ induce a (positive definite) scalar product on $V^R$, again denoted by $\<\cdot,\cdot\>$.

Let $D : C^\infty(M;V^R) \to C^\infty(M;V^L)$ be the Dirac operator acting on right-handed twisted spinors. 
In terms of a local Lorentz-orthonormal tangent frame $e_0,e_1,\ldots,e_n$ this operator is given by
$$
D= \sum_{j=0}^n \ep_j\gamma(e_j)\nabla_{e_j}
$$
where $\ep_j=\<e_j,e_j\>=\pm 1$ and $\nabla$ is the connection on $V^R$ induced by the Levi-Civita connection on $S^RM$ and $\nE$.
The Clifford multiplication $\gamma(X)$ of a tangent vector $X\in T_pM$ on a twisted spinor $\phi\otimes e\in V^R_p=S^R_pM\otimes E_p$ is to be understood as acting on the first factor, $\gamma(X)(\phi\otimes e)=(\gamma(X)\phi)\otimes e$.
The Dirac operator is a hyperbolic linear differential operator of first order.

Along a smooth spacelike hypersurface $\Sigma\subset M$ with past-directed unit normal field $\nu$ the Dirac operator can be written as
\begin{equation}
D = -\gamma(\nu) \left(\nabla_{\nu} + \rmi\, A_\Sigma - \frac{n}{2}H \right)
\label{eq:DiracSurface+}
\end{equation}
where $H$ is the mean curvature of $\Sigma$ with respect to $\nu$ and $A_\Sigma$ is the elliptic twisted Dirac operator of the Riemannian manifold $\Sigma$.

\subsection{The Cauchy problem}
Now we fix two smooth spacelike Cauchy hypersurfaces $\Sigma_1$ and $\Sigma_0$ in $M$.
We assume that $\Sigma_0$ lies in the chronological past of $\Sigma_1$.
Then we consider the region $M_0$ ``in between'' $\Sigma_0$ and $\Sigma_1$, more precisely, $M_0=J^+(\Sigma_0)\cap J^-(\Sigma_1)$ where $J^-$ and $J^+$ denote the causal past and future, respectively.
Since $M$ is spatially compact, the region $M_0$ is a compact manifold with boundary, the boundary being the disjoint union of $\Sigma_0$ and $\Sigma_1$.

For any compact spacelike hypersurface $\Sigma\subset M$ we define the $L^2$-scalar product for $u,v\in C^\infty(\Sigma,V^R)$ by 
$$
(u,v)_{L^2} = \int_\Sigma \<u,v\>\, \dA
$$
where $\dA$ denotes the volume element of $\Sigma$ induced by its Riemannian metric.
Recall that the inner product $\<\cdot,\cdot\>$ on $V^R$ is positive definite.
The completion of $C^\infty(\Sigma,V^R)$ w.r.t.\ the $L^2$-norm will be denoted by $L^2(\Sigma,V^R)$.

Similarly, using an auxiliary positive definite scalar product on $V^L$ we can define $L^2(M_0,V^L)$ where we integrate against the volume element induced by the Lorentzian metric on $M$.
By compactness of $M_0$ different choices of auxiliary scalar products on $V^L$ will give rise to equivalent $L^2$-norms.
Hence $L^2(M_0,V^L)$ is unanimously defined as a topological vector space.

Finally, we complete $C^\infty(M_0,V^R)$ w.r.t.\ the $L^2$-graph-norm of $D$
$$
\|u\|_\FE^2 = \|u\|_{L^2}^2 + \|Du\|_{L^2}^2
$$
and obtain the ``finite-energy'' space $\FE(M_0,V^R)$.

Now the Dirac operator obviously extends to a bounded operator $D:\FE(M_0,V^R)\to L^2(M_0,V^L)$.
It can be checked that the restriction map $C^\infty(M_0,V^R)\to C^\infty(\Sigma,V^R)$ extends uniquely to a bounded operator $\res_\Sigma:\FE(M_0,V^R)\to L^2(\Sigma,V^R)$ if $\Sigma$ is a spacelike Cauchy hypersurface. 

In these function spaces the Cauchy problem is well posed:

\begin{theorem}\label{thm:InhomoCauchyProblem}
Let $\Sigma\subset M_0$ be a smooth spacelike Cauchy hypersurface.
Then the mapping
\begin{align*}
\res_\Sigma \oplus D : \FE(M_0,V^R) &\to L^2(\Sigma;V^R)\oplus L^2(M_0,V^L)
\end{align*}
is an isomorphism of Hilbert spaces.\qed
\end{theorem}

In particular, we get well-posedness of the Cauchy problem for the homogeneous Dirac equation:

\begin{cor}\label{cor:HomoCauchyProblem}
For any smooth spacelike Cauchy hypersurface $\Sigma\subset M_0$ the restriction mapping
$$
\res_\Sigma : \{u\in\FE(M_0,V^R)\mid Du=0\} \to L^2(\Sigma,V^R)
$$
is an isomorphism of Hilbert spaces.\qed
\end{cor}

For details see \cite{BS15}.

\subsection{The wave evolution}
Applying Corollary~\ref{cor:HomoCauchyProblem} to $\Sigma=\Sigma_0$ and to $\Sigma=\Sigma_1$ we can define the \emph{wave evolution operator} 
\begin{equation}
Q:L^2(\Sigma_0,V^R) \to L^2(\Sigma_1,V^R)
\label{eq:Q}
\end{equation}
by the commutative diagram
$$
\xymatrix{
&  \{u\in\FE(M_0,V^R)\mid Du=0\} \ar[dr]^{\res_{\Sigma_1}}_\cong \ar[dl]_{\res_{\Sigma_0}}^\cong &  \\
L^2(\Sigma_0,V^R) \ar[rr]^{Q} & & L^2(\Sigma_1,V^R)
}
$$

By construction, $Q$ is an isomorphism.
One can check that $Q$ is actually unitary, that it restricts to isomorphisms $H^s(\Sigma_0,V^R)\to H^s(\Sigma_1,V^R)$ for all $s>0$ and that it extends to  isomorphisms $H^s(\Sigma_0,V^R)\to H^s(\Sigma_1,V^R)$ for all $s<0$.
Here $H^s$ denote the corresponding Sobolev spaces. 
As a consequence, $Q$ maps $C^\infty(\Sigma_0,V^R)$ to $C^\infty(\Sigma_1,V^R)$.

In fact, well-posedness of the Dirac equation also holds for smooth sections, i.e.,
\begin{equation}
\res_\Sigma \oplus D : C^\infty(M_0,V^R) \to C^\infty(\Sigma;V^R)\oplus C^\infty(M_0,V^L)
\label{eq:GlattWohlgestellt}
\end{equation}
is an isomorphism of Fr\'echet spaces.

\section{Fredholm pairs}
\label{sec:Fredholm}

In this section we collect a few functional-analytic facts which will be useful later.

\begin{definition}
Let $H$ be a Hilbert space and let $B_0,B_1\subset H$ be closed linear subspaces.
Then $(B_0,B_1)$ is called a \emph{Fredholm pair} if $B_0\cap B_1$ is finite dimensional and $B_0+B_1$ is closed and has finite codimension.
The number
$$
\ind (B_0,B_1)=\dim (B_0\cap B_1)-\dim (H/ (B_0+B_1))
$$
is called the \emph{index} of the pair $(B_0,B_1)$.
\end{definition}

We list a few elementary properties of Fredholm pairs.
For details see \cite[Ch.~IV, \S~4]{K95}.

\begin{remark}\label{rem:Fredpairs}
\begin{enumerate}[1.)]
\item\label{Fredpair1}
The pair $(B_0,B_1)$ is Fredholm if and only if $(B_1,B_0)$ is a Fredholm pair and in this case 
$$
\ind(B_0,B_1)=\ind(B_1,B_0).
$$
\item\label{Fredpair2}
The pair $(B_0,B_1)$ is Fredholm if and only if $(B_0^\perp,B_1^\perp)$ is a Fredholm pair and in this case 
$$
\ind(B_0,B_1)=-\ind(B_0^\perp,B_1^\perp).
$$
\item\label{Fredpair3}
Let $B_0'\subset H$ be a closed linear subspace with $B_0\subset B_0'$ and $\dim(B_0'/B_0)<\infty$.
Then $(B_0,B_1)$ is a Fredholm pair if and only if $(B_0',B_1)$ is a Fredholm pair and in this case
$$
\ind (B_0',B_1) = \ind (B_0,B_1) +\dim(B_0'/B_0).
$$
\end{enumerate}
\end{remark}

The following lemma reformulates the concept of Fredholm pairs in terms of orthogonal projections. 
For a proof see e.g.\ \cite[Lemma~24.3]{BW93}.
Here and henceforth, the orthogonal projection onto a closed subspace $V$ of a Hilbert space $H$ will be denoted by 
$$
\pi_V:H\to V.
$$

\begin{lem}\label{projections}
Let $B_0,B_1\subset H$ be closed linear subspaces. 
Then $(B_0,B_1)$ is a Fredholm pair of index $k$ if and only if
$$
\pi_{B_1^\perp}|_{B_0}:B_0\rightarrow B_1^\perp
$$
is a Fredholm operator of index $k$. 
In this case we have $\ker \big(\pi_{B_1^\perp}|_{B_0}\big)=B_0\cap B_1$ and $\coker \big(\pi_{B_1^\perp}|_{B_0}\big)\cong B_0^\perp\cap B_1^\perp$.\qed
\end{lem}

Let $E$, $F$, $H_0$, and $H_1$ be Hilbert spaces and let $L:E\to F$, $r_0:E\to H_0$, and $r_1:E\to H_1$ be bounded linear maps.
We assume that 
$$
r_j\oplus L : E\to H_j \oplus F
$$
is an isomorphism for $j=0$ and for $j=1$.
Then $r_j$ restricts to an isomorphism $\ker(L)\to H_j$.
We define the isomorphism $Q:H_0\to H_1$ by the commutative diagram
$$
\xymatrix{
&  \ker(L) \ar[dr]^{r_1}_\cong \ar[dl]_{r_0}^\cong &  \\
H_0 \ar[rr]^{Q} & & H_1
}
$$

\begin{proposition}
\label{prop:FredholmAbstrakt}
Assume that $r_0\oplus r_1:E\to H_0\oplus H_1$ is onto.
Let $B_j\subset H_j$ be closed linear subspaces.
Then the following are equivalent:
\begin{enumerate}[(i)]
\item\label{Fred0}
The pair $(B_0,Q^{-1}B_1)$ is Fredholm of index $k$;
\item\label{Fred1}
The pair $(QB_0,B_1)$ is Fredholm of index $k$;
\item\label{Fred2}
The operator $(\pi_{B_0^\perp}\circ r_0)\oplus (\pi_{B_1^\perp}\circ r_1)\oplus L : E \to B_0^\perp \oplus B_1^\perp \oplus F$ is Fredholm of index~$k$;
\item\label{Fred3}
The restriction $L:\ker(\pi_{B_0^\perp}\circ r_0)\cap \ker(\pi_{B_1^\perp}\circ r_1)\to F$ is a Fredholm operator of index~$k$.
\end{enumerate}
\end{proposition}

\begin{proof}
a)
The equivalence of \eqref{Fred0} and \eqref{Fred1} is clear because $Q:H_0\to H_1$ is an isomorphism.

b)
Proposition~A.1.(iv) in \cite{BB12} with $P=L$ states that \eqref{Fred2} is equivalent to the operator
$$
((\pi_{B_0^\perp}\circ r_0)\oplus (\pi_{B_1^\perp}\circ r_1))\vert_{\ker L} : \ker L\to B_0^\perp \oplus B_1^\perp
$$
being Fredholm of index $k$. Since $r_1\vert_{\ker L}=Q\circ r_0\vert_{\ker L}$ this is again equivalent to 
\begin{align*}
\widetilde{L}:H_0\to B_0^\perp\oplus B_1^\perp,
\quad
h_0\mapsto (\pi_{B_0^\perp}(h_0),\pi_{B_1^\perp}(Q(h_0)),
\end{align*}
being Fredholm of index $k$. 
With respect to the splittings $\widetilde{L}:H_0=B_0^\perp\oplus B_0\to B_0^\perp\oplus B_1^\perp$ takes the operator matrix form
$$
\widetilde{L} =
\begin{pmatrix}
1_{B_0^\perp} & 0 \\
\pi_{B_1^\perp}Q|_{B_0^\perp} & \pi_{B_1^\perp}Q|_{B_0}
\end{pmatrix}.
$$
Now $\widetilde{L}$ is Fredholm if and only if it is invertible modulo compact operators.
This is equivalent to $\pi_{B_1^\perp}Q|_{B_0}$ being invertible modulo compact operators, i.e.\ being Fredholm.
By deformation invariance of the index we have
$$
\ind\big(\widetilde{L}\big)
=
\ind
\begin{pmatrix}
1_{B_0^\perp} & 0 \\
\pi_{B_1^\perp}Q|_{B_0^\perp} & \pi_{B_1^\perp}Q|_{B_0}
\end{pmatrix}
=
\ind
\begin{pmatrix}
1_{B_0^\perp} & 0 \\
0 & \pi_{B_1^\perp}Q|_{B_0}
\end{pmatrix}
=
\ind \big(\pi_{B_1^\perp}Q|_{B_0}\big).
$$
This shows that \eqref{Fred2} is equivalent to $\pi_{B_1^\perp}\vert_{QB_0}$ being a Fredholm operator of index~$k$.
The equivalence of \eqref{Fred1} and \eqref{Fred2} now follows with Lemma~\ref{projections}.

c)
The equivalence of \eqref{Fred2} and \eqref{Fred3} is Proposition~A.1.(iv) in \cite{BB12} with 
$$
P=(\pi_{B_0^\perp}\circ r_0)\oplus (\pi_{B_1^\perp}\circ r_1):E\to B_0^\perp \oplus B_1^\perp .
$$
Note that $P$ is onto because $r_0\oplus r_1:E\to H_0\oplus H_1$ is onto by assumption.
\end{proof}

\section{Boundary value problems for the Dirac operator}
\label{sec:BVP}

We return to our twisted Dirac operator $D$ on a Lorentzian manifold as introduced in Section~\ref{sec:setup}.
Let $B_0\subset L^2(\Sigma_0,V^R)$ and $B_1\subset L^2(\Sigma_1,V^R)$ be closed subspaces.
Denote by $Q:L^2(\Sigma_0,V^R) \to L^2(\Sigma_1,V^R)$ the wave evolution operator as defined in \eqref{eq:Q}.
Combining Theorem~\ref{thm:InhomoCauchyProblem} and Proposition~\ref{prop:FredholmAbstrakt} we get

\begin{theorem}
\label{thm:DiracFredholm}
The following are equivalent:
\begin{enumerate}[(i)]
\item\label{Dirac0}
The pair $(B_0,Q^{-1}B_1)$ is Fredholm of index $k$;
\item\label{Dirac1}
The pair $(QB_0,B_1)$ is Fredholm of index $k$;
\item\label{Dirac2}
The operator 
$$
(\pi_{B_0^\perp}\circ \res_{\Sigma_0})\oplus (\pi_{B_1^\perp}\circ \res_{\Sigma_1})\oplus D : \FE(M_0,V^R) \to B_0^\perp \oplus B_1^\perp \oplus L^2(M_0,V^L)
$$
is Fredholm of index~$k$;
\item\label{Dirac3}
The restriction 
$$
D:\ker(\pi_{B_0^\perp}\circ \res_{\Sigma_0})\cap \ker(\pi_{B_1^\perp}\circ \res_{\Sigma_1})\to L^2(M_0,V^L)
$$ 
is a Fredholm operator of index~$k$.\qed
\end{enumerate}
\end{theorem}

\begin{definition}
If these conditions hold then we call $(B_0,B_1)$ a \emph{Dirac-Fredholm pair} of index $k$.
\end{definition}

Condition \eqref{Dirac3} in Theorem~\ref{thm:DiracFredholm} means that we consider the Dirac equation $Du=f$ subject to the boundary conditions $u|_{\Sigma_0}\in B_0$ and $u|_{\Sigma_1}\in B_1$.
We now look at concrete examples.

\subsection{(Anti) Atiyah-Patodi-Singer boundary conditions}
For any subset $I\subset\R$ denote by $\chi_I:\R\to\R$ its characteristic function.
Denote the Riemannian Dirac operators for the boundary parts by $A_0:=A_{\Sigma_0}$ and $A_1:=A_{\Sigma_1}$, compare \eqref{eq:DiracSurface+}.
Since $A_0$ and $A_1$ are self-adjoint elliptic operators on closed Riemannian manifolds they have real and discrete spectrum.
We consider the spectral projector $\chi_I(A_0):L^2(\Sigma_0,V^R)\to L^2(\Sigma_0,V^R)$ and similarly for $A_1$.
This is the orthogonal projector onto the sum of the $A_0$-eigenspaces to all eigenvalues contained in $I$.
Its range will be denoted by $L^2_I(A_0):=\chi_I(A_0)(L^2(\Sigma_0,V^R))\subset L^2(\Sigma_0,V^R)$ and similarly for $A_1$.

Now the \emph{Atiyah-Patodi-Singer boundary conditions} correspond to the choice $B_0=L^2_{(-\infty,0)}(A_0)$ and $B_1=L^2_{(0,\infty)}(A_1)$.
It was shown in \cite[Thm.~7.1]{BS15} that the Dirac operator subject to these boundary conditions is Fredholm.
In other words, the pair $(B_0,B_1) = (L^2_{(-\infty,0)}(A_0),L^2_{(0,\infty)}(A_1))$ is Fredholm and its index is given by
$$
\ind(B_0,B_1)
=
\int_{M_0} \Adach (\nabla)\wedge \ch(\nabla^E) + \int_{\partial M_0}\mathcal{T} - \frac{h(A_{0})+h(A_{1})+\eta(A_{0})-\eta(A_{1})}{2}\, .
$$
Here $\Adach (\nabla)$ is the $\Adach$-form built from the curvature of the Levi-Civita connection $\nabla$ on $M$ and $\ch(\nabla^E)$ is the Chern character form of the curvature of $\nabla^E$.
The form $\mathcal{T}$ is the corresponding transgression form.
In particular, the boundary integral vanishes if the given metric and connections have product structure near the boundary.
Moreover, $h(A)$ denotes the dimension of the kernel of an operator $A$ and $\eta(A)$ its $\eta$-invariant.
See \cite{BS15} for details.

By Remark~\ref{rem:Fredpairs}.\ref{Fredpair2} the complementary boundary conditions $(\APS_0(0)^\perp,\APS_1(0)^\perp) = (L^2_{[0,\infty)}(A_0),L^2_{(-\infty,0]}(A_1))$, the \emph{anti-Atiyah-Patodi-Singer boundary conditions}, are also Fredholm and the index has the opposite sign,
$$
\ind(B_0^\perp,B_1^\perp)
=
-\int_{M_0} \Adach (\nabla)\wedge \ch(\nabla^E) - \int_{\partial M_0}\mathcal{T} + \frac{h(A_{0})+h(A_{1})+\eta(A_{0})-\eta(A_{1})}{2}\, .
$$

\subsection{Generalized Atiyah-Patodi-Singer boundary conditions}\label{ssec:gaps}
For $a\in\R$ we define $\APS_0(a):=L^2_{(-\infty,a)}(A_0)\subset L^2(\Sigma_0,V^R)$ and $\APS_1(a):=L^2_{(a,\infty)}(A_1)\subset L^2(\Sigma_1,V^R)$.
Since cutting the spectrum of the boundary operator at $0$ is somewhat arbitrary we may want to fix $a_0, a_1\in\R$ and consider the boundary conditions $(B_0,B_1) = (\APS_0(a_0), \APS_1(a_1))$.
These boundary conditions are known as \emph{generalized Atiyah-Patodi-Singer boundary conditions}.
Since all eigenvalues of the Riemannian Dirac operators $A_{\Sigma_i}$ are of finite multiplicity, these boundary conditions differ from the Atiyah-Patodi-Singer boundary conditions only by finite dimensional spaces, more precisely,
\begin{align*}
\APS_0(a_0)/\APS_0(0)
&\cong 
L^2_{[0,a_0)}(A_0)\hspace{5pt} \mathrm{for}\hspace{2pt} a_0\geq 0,\\
\APS_0(0)/\APS_0(a_0)
&\cong 
L^2_{[a_0,0)}(A_0)\hspace{5pt} \mathrm{for}\hspace{2pt} a_0<0,
\end{align*}
and similarly for $\APS_1$.
Remark~\ref{rem:Fredpairs}.\ref{Fredpair3} implies that generalized Atiyah-Patodi-Singer boundary conditions also form a Dirac-Fredholm pair.
Setting
$$
W_0:=\begin{cases}
        L^2_{[0,a_0)}(A_0), & a_0\geq 0,
        \\
        L^2_{[a_0,0)}(A_0), & a_0<0,
        \end{cases}
\quad\mbox{ and }\quad
W_1:=\begin{cases}
        L^2_{(0,a_1]}(A_1), & a_1\geq 0,
        \\
        L^2_{(a_1,0]}(A_1), & a_1<0,
        \end{cases}
$$
we get for their index
\begin{align*}
\ind (\APS_0(a_0),&\APS_1(a_1)) \\
&=
\ind (\APS_0(0),\APS_1(0))+\mathrm{sgn}(a_0)\dim (W_0) -\mathrm{sgn}(a_1)\dim (W_1) \\
&=
\int_{M_0} \Adach (\nabla)\wedge \ch(\nabla^E) + \int_{\partial M_0}\mathcal{T} - \frac{h(A_{0})+h(A_{1})+\eta(A_{0})-\eta(A_{1})}{2}\\
&\quad +\mathrm{sgn}(a_0)\dim (W_0) -\mathrm{sgn}(a_1)\dim (W_1).
\end{align*}
In other words, the correction terms in the index formula are given by the total multiplicity of the eigenvalues of $A_j$ between $0$ and $a_j$.

\subsection{Boundary conditions in graph form}\label{sec:graph}
In the previous section we discussed certain finite dimensional modifications of the Atiyah-Patodi-Singer boundary conditions, which lead to corrections in the index formula.
Now we introduce continuous deformations of Dirac-Fredholm pairs, leaving the index unchanged.
Formally, the definition coincides with that of $D$-elliptic boundary conditions for the elliptic Dirac operator on \emph{Riemannian} manifolds as introduced in \cite[Def.~7.5 and Thm.~7.11]{BB12}.

\begin{definition}\label{def:ellbc}
We call a pair $(B_0,B_1)$ of closed subspaces $B_i\subset L^2(\Sigma_i,V^R)$ \emph{boundary conditions in graph form} if there are $L^2$-orthogonal decompositions
\begin{equation}
L^2(\Sigma_i,V^R)=V_i^-\oplus W_i^-\oplus V_i^+\oplus W_i^+, \quad\quad i=0,1,
\label{eq:Dhyp}
\end{equation}
such that
\begin{enumerate}[(i)]
\item $W_i^+, W_i^-$ are finite dimensional;
\item $W_i^-\oplus V_i^-=L^2_{(-\infty,a_i)}(A_i)$ and $W_i^+\oplus V_i^+=L^2_{[a_i,\infty)}(A_i)$ for some $a_i\in\mathbb{R}$;
\item There are bounded linear maps $g_0:V_0^-\to V_0^+$ and $g_1:V_1^+\to V_1^-$ such that
\begin{align*}
B_0&=W_0^+\oplus\Gamma(g_0),\\
B_1&=W_1^-\oplus\Gamma(g_1),
\end{align*}
where $\Gamma(g_{0/1}):=\lbrace v+g_{0/1}v\hspace{2pt}\vert\hspace{2pt}v\in V_{0/1}^\mp\rbrace$ denotes the graph of $g_{0/1}$.
\end{enumerate}
\end{definition}

\begin{remark}
In the setting of Definition \ref{def:ellbc} we have
$$
V_i^-\oplus V_i^+=\Gamma(g_i)\oplus\Gamma(-g_i\adj)
$$
where both decompositions are orthogonal. 
With respect to the splitting $V_i^-\oplus V_i^+$ the projections onto $\Gamma(g_i)$ are given by
\begin{equation}
\begin{pmatrix}
\id & 0 \\ 
g_i & 0
\end{pmatrix} 
\begin{pmatrix}
\id & -g_i\adj \\ 
g_i & \id
\end{pmatrix} ^{-1}
=
\begin{pmatrix}
(\id +g_i\adj g_i)^{-1} & (\id +g_i\adj g_i)^{-1}g_i\adj \\ 
g_i(\id +g_i\adj g_i )^{-1} & g_i(\id +g_i\adj g_i)^{-1}g_i\adj
\end{pmatrix} 
\label{eq:GraphProj}
\end{equation}
see \cite[Lemma~7.7~and~Remark~7.8]{BB12}.
Here $g_i\adj:V_i^\pm \to V_i^\mp$ denotes the adjoint linear map.
\end{remark}

The next Lemma shows that deforming Atiyah-Patodi-Singer boundary conditions for the Dirac operator to a graph preserves Fredholm property as well as the index.

\begin{proposition}\label{prop:graph}
Let $a_0,a_1\in\R$ be given.
Then there exists an $\eps>0$ such that for any bounded linear maps $g_i:\APS_i(a_i)\to \APS_i(a_i)^\perp$ the pair $(\Gamma(g_0),\Gamma(g_1))$ is a Dirac-Fredholm pair of the same index as $(\APS_0(a_0),\APS_1(a_1))$ provided 
\begin{enumerate}[(A)]
\item\label{g0g1kompakt}
$g_0$ or $g_1$ is compact or
\item\label{g0g1klein} 
$\|g_0\|\cdot\|g_1\|<\eps$.
\end{enumerate}
\end{proposition}

\begin{proof}
We put $\gt_1:=Q^{-1}\circ g_1 \circ Q$.
Now $(\Gamma(g_0),\Gamma(g_1))$ is a Dirac-Fredholm pair of index $k$ if and only if $(\Gamma(g_0),Q^{-1}\Gamma(g_1))=(\Gamma(g_0),\Gamma(\gt_1))$ is a Fredholm pair of index $k$.
By Lemma \ref{projections} this is equivalent to
$$
\pi_{\Gamma(g_0)}\vert_{\Gamma(-\gt_1\adj)}:\Gamma(-\gt_1\adj)\to\Gamma(g_0)
$$
being a Fredholm operator of index $-k$. 
Since the maps $(\id+g_i):\APS_i(a_i)\to\Gamma(g_i)$ and $(\id-g_i\adj):\APS_i(a_i)^\perp\to\Gamma(-g_i\adj)$ are isomorphisms, this is equivalent to
$$
\xymatrix{
L:Q^{-1}\APS_1(a_1)^\perp \ar[r]^{\quad\quad\cong} &\Gamma(-\gt_1^\perp)\ar[r]^{\pi_{\Gamma(g_0)}}&\Gamma(g_0)\ar[r]^{\cong\quad}& \APS_0(a_0)
}
$$
being Fredholm of index $-k$.
We write $B_0:=\APS_0(a_0)$ and $B_1:=Q^{-1}\APS_1(a_1)$.
Using \eqref{eq:GraphProj} we see that $L$ takes the form
\begin{align*}
L
&=
\begin{pmatrix}
\id & 0
\end{pmatrix} 
\begin{pmatrix}
(\id+g_0\adj g_0)^{-1} & (\id+g_0\adj g_0)^{-1}g_0\adj \\ 
g_0(\id+g_0\adj g_0)^{-1} & g_0(\id+g_0\adj g_0)^{-1}g_0\adj
\end{pmatrix} 
\begin{pmatrix}
\pi_{B_0}\vert_{B_1} & \pi_{B_0}\vert_{B_1^\perp} \\ 
\pi_{B_0^\perp}\vert_{B_1} & \pi_{B_0^\perp}\vert_{B_1^\perp}
\end{pmatrix} 
\begin{pmatrix}
-\gt_1\adj \\ 
\id
\end{pmatrix} \\
&=
\begin{pmatrix}
(\id+g_0\adj g_0)^{-1} & (\id+g_0\adj g_0)^{-1}g_0\adj
\end{pmatrix} 
\begin{pmatrix}
-(\pi_{B_0}\vert_{B_1}) \gt_1\adj + \pi_{B_0}\vert_{B_1^\perp} \\ 
-(\pi_{B_0^\perp}\vert_{B_1})\gt_1\adj + \pi_{B_0^\perp}\vert_{B_1^\perp}
\end{pmatrix} \\
&=
(\id+g_0\adj g_0)^{-1}(-(\underbrace{\pi_{B_0}\vert_{B_1}}_{=:\alpha}) \gt_1\adj
+\underbrace{\pi_{B_0}\vert_{B_1^\perp}}_{=:\beta}
-\underbrace{g_0\adj(\pi_{B_0^\perp}\vert_{B_1})\gt_1\adj}_{=:\gamma}
+g_0\adj(\underbrace{\pi_{B_0^\perp}\vert_{B_1^\perp}}_{=:\delta})).
\end{align*}
The operators $\alpha$ and $\delta$ are compact by \cite[Lemma~2.6]{BS16}.
If $g_0$ or $g_1$ (and hence $g_0\adj$ or $\gt_1\adj$) is compact then $\gamma$ is compact as well.
Since $\beta=\pi_{B_0}\vert_{B_1^\perp}$ is a Fredholm operator of index $\ind(B_0^\perp,B_1^\perp)=-\ind(B_0,B_1)$ the same is true for $L$.

If $\|g_0\|\cdot\|g_1\|<\eps$ then 
$$
\|\gamma\| 
\le 
\|g_0\adj\|\cdot\|\pi_{B_0^\perp}\vert_{B_1}\|\cdot\|\gt_1\adj\|
=
\|g_0\|\cdot\|\pi_{B_0^\perp}\vert_{B_1}\|\cdot\|g_1\|
<
\eps
$$
is small so that $\beta-\gamma$ (and hence $L$) is again a Fredholm operator of the same index as $\beta$.
\end{proof}

\begin{remark}
In Example~\ref{exa:statisch} we will see that conditions \eqref{g0g1kompakt} and \eqref{g0g1klein} cannot be dropped.
Without these assumptions boundary conditions in graph form do not give rise to a Fredholm operator in general.
\end{remark}

\begin{remark}
Proposition~\ref{prop:graph} can be applied in particular if $g_0=0$ or $g_1=0$.
Thus in the setting of the lemma $(\Gamma(g_0),\APS_1(a_1))$ and $(\APS_0(a_0),\Gamma(g_1))$ are Dirac-Fredholm pairs of index $\ind(\APS_0(a_0),\APS_1(a_1))$.
\end{remark}

If $g:B\to B^\perp$ is a bounded linear map and $\widetilde{B}\subset B$, then $\Gamma(g)/\Gamma(g\vert_{\widetilde{B}})\cong B/\widetilde{B}$. 
Combining Section~\ref{ssec:gaps} and Proposition~\ref{prop:graph} we get the following result.

\begin{cor}
Let $(B_0,B_1)$ be boundary conditions in graph form with $g_0$ or $g_1$ compact or $\|g_0\|\cdot\|g_1\|$ sufficiently small.
Then $(B_0,B_1)$ is a Dirac-Fredholm pair and its index is given by
$$
\ind(B_0,B_1)=\ind(L^2_{(-\infty,a_0)}(A_0),L^2_{[a_1,\infty)}(A_1))+\dim(W_0^+)-\dim (W_0^-)+\dim (W_1^-)-\dim(W_1^+).
$$
\end{cor}

\subsection{Local boundary conditions}

Suppose we have subbundles $E_j\subset V^R|_{\Sigma_j}$.
Then the boundary condition $B_j=L^2(\Sigma_j,E_j)\subset L^2(\Sigma_j,V^R)$ is called a \emph{local boundary condition}.
The following Lemma relates local boundary conditions to graph deformations of generalized Atiyah-Patodi-Singer boundary conditions as discussed in Section \ref{sec:graph}.

\begin{proposition}\label{prop:graphs}
Let $E_j\subset V^R\vert_{\Sigma_j}$ be a subbundle and $B_j=L^2(\Sigma_j,E_j)$.
Then the following are equivalent:
\begin{enumerate}[(i)]
 \item\label{prop:graph1} 
 There exists an $L^2$-orthogonal splitting 
 $$
 L^2(\Sigma_j,V^R)=V_j^-\oplus W_j^-\oplus V_j^+\oplus W_j^+
 $$
 and bounded linear maps $g_j:V_j^\mp\to V_j^\pm$ as in Definition~\ref{def:ellbc} such that $B_j=W_j^\pm\oplus\Gamma(g_j)$;
 \item\label{prop:graph2} 
 For every $x\in\Sigma_j$ and $\xi\in T_x\Sigma_j$, $\xi\neq 0$, the projection $\pi_{(E_j)_x}:(V^R\vert_{\Sigma_j})_x\to (E_j)_x$ restricts to an isomorphism from the sum of eigenspaces to the negative (for $j=0$) or positive (for $j=1$) eigenvalues of $\rmi \sigma_{A_j}(\xi)$ onto $(E_j)_x$.
\end{enumerate}
\end{proposition}

\begin{proof}
First note that the fiberwise projections $\pi_{(E_j)_x}$ yield a projection map 
$$
P_j:L^2(\Sigma_j,V^R)\to L^2(\Sigma_j,E_j).
$$
Then \eqref{prop:graph2} is equivalent to the operator
$$
P_j-\pi_{\APS_j(a)^\perp}:L^2(\Sigma_j,V^R)\to L^2(\Sigma_j,V^R)
$$
being Fredholm for some (and then all) $a\in\mathbb{R}$, as stated in \cite[Thm.~7.20]{BB12}.

First we show that \eqref{prop:graph1} implies \eqref{prop:graph2}.
Since the sum of a Fredholm operator and a finite rank operator is again Fredholm, we can assume that $W_j^-=W_j^+=0$, $V_j^-=\APS_j(0)$ and $V_j^+=\APS_j(0)^\perp$.
With respect to the splitting $L^2(\Sigma_j,V^R)=V_j^-\oplus V_j^+$ we then have by \eqref{eq:GraphProj}
$$
\pi_{\APS_j(0)^\perp}=\begin{pmatrix}
0 & 0 \\ 
0 & \id
\end{pmatrix}
\quad\mbox{ and }\quad
P_j=
\begin{pmatrix}
\id & 0 \\ 
g_j & 0
\end{pmatrix} 
\begin{pmatrix}
\id & -g_j\adj \\ 
g_j & \id
\end{pmatrix} ^{-1}
$$
and hence
$$
P_j-\pi_{\APS_j(0)^\perp}=
\begin{pmatrix}
\id & 0 \\ 
0 & -\id
\end{pmatrix} 
\begin{pmatrix}
\id & -g_j\adj \\ 
g_j & \id
\end{pmatrix} ^{-1}
$$
is an isomorphism.

Now we show that (\ref{prop:graph2}) implies (\ref{prop:graph1}). 
We construct the decomposition in (\ref{prop:graph1}) and the map $g_j$ for $j=0$.
The case $j=1$ is analogous with the labels $+$ and $-$ interchanged.

Assuming that $P-\pi_{\APS_0(0)^\perp}$ is Fredholm we set
\begin{align*}
W_0^+&:=B_0\cap \APS_0(0) ^\perp,& 
V_0^+&:=(W_0^+)^\perp\cap  \APS_0(0)^\perp ,\\
V_0^-&:=\pi_{\APS_0(0)}(B_0),& 
W_0^-&:=(V_0^-)^\perp\cap\APS_0(0).
\end{align*}
We then have $V_0^+\oplus W_0^+=L^2_{[0,\infty)}(\Sigma_0,V^R)$ and $V_0^-\oplus W_0^-=L^2_{(-\infty,0)}(\Sigma_0,V^R)$.
Furthermore, for $w\in W_0^+$ we have 
$$
(P_0-\pi_{\APS_0(0)^\perp})(w)=w-w=0,
$$
so $W_0^+$ is contained in the kernel of $P_0-\pi_{\APS_0(0)^\perp}$ and is hence finite dimensional.
For $w\in W_0^-$ we observe
$$
( w,P_0x-\pi_{\APS_0(0)^\perp }x )_{L^2}
=
(w,\pi_{\APS_0(0)}P_0x)_{L^2} - ( w,\pi_{\APS_0(0)^\perp }x)_{L^2}
=0-0=0
$$ 
for all $x\in L^2(\Sigma_0,V^R)$.
Thus $W_0^-$ is contained in the orthogonal complement of the range of $P_0-\pi_{\APS_0(0)^\perp}$ and is hence finite dimensional.
 
Setting $U_0:=(W_0^+)^\perp\cap B_0$ we have $B_0=W_0^+\oplus U_0$ and $\pi_{\APS_0(0)}|_{U_0}:U_0\to V_0^-$ is an isomorphism.
Now the composition
$$
\xymatrix{
g_0:V_0^- \ar[rr]^{\quad(\pi_{\APS_0(0)}|_{U_0})^{-1}} &&U_0\ar[rr]^{\pi_{\APS_0(0)^\perp }}&&V_0^+
}
$$
is a bounded linear map with $U_0=\Gamma(g_0)$.
Then we have $B_0=W_0^+\oplus\Gamma(g_0)$ which concludes the proof.
\end{proof}
\begin{example}
Let $G$ be a smooth field of unitary involutions of $V^R$ along $\Sigma_0\cup \Sigma_1$ which anti-commutes with $A_j$, $j=0,1$.
Then $G$ anti-commutes with all $\sigma_{A_j}(\xi)$, $\xi\in T^*\Sigma_0\cup T^*\Sigma_1$.
%
%
We split $V^R\vert_{\Sigma_j}=E_j^+\oplus E_j^-$ into the sum of $(\pm 1)$-eigenspaces for $G$.

Let $\xi\in T_x\Sigma_j$, $\xi\neq 0$.
W.l.o.g.\ we assume $|\xi|=1$.
Then $(\rmi \sigma_{A_j}(\xi))^2=1$ and $\rmi \sigma_{A_j}(\xi)$ has only the eigenvalues $\pm 1$.
Since $G$ and $\sigma_{A_j}(\xi)$ anti-commute $(E_j^+)_x$, $(E_j^-)_x$, and the $(\pm 1)$-eigenspaces for $\rmi \sigma_{A_j}(\xi)$ all have the same dimension, namely half the dimension of $(V^R)_x$.
The projections onto $E_j^\pm$ are given by $\pi_{(E_j^\pm)_x} = \frac12(1\pm G)$.

Now assume that $\rmi \sigma_{A_j}(\xi)u=-u$ and that $\pi_{(E_j^\pm)_x} u=0$.
Then
$$
0
=\rmi \sigma_{A_j}(\xi)\pi_{(E_j^\pm)_x} u 
= \pi_{(E_j^\mp)_x} \rmi \sigma_{A_j}(\xi) u
=-\pi_{(E_j^\mp)_x} u
$$
and hence
$$
u=\pi_{(E_j^\pm)_x} u + \pi_{(E_j^\mp)_x} u = 0.
$$
Thus both $\pi_{(E_j^+)_x}$ and $\pi_{(E_j^-)_x}$ are injective when restricted to the $(-1)$-eigenspace of $\rmi \sigma_{A_j}(\xi)$.
For dimensional reasons these restrictions form an isomorphism onto $(E_j^\pm)_x$.
Thus Proposition~\ref{prop:graphs} applies to both $E_j=E_j^+$ and $E_j=E_j^-$.
Hence $(L^2(\Sigma_0,E_0^\pm),L^2(\Sigma_1,E_1^\pm))$ are boundary conditions in graph form.
These boundary conditions are known as \emph{chirality conditions}.

The description of the chirality boundary conditions in Definition~\ref{def:ellbc} is as follows:
$V_0^-=L^2_{(-\infty,0)}(A_0)$, $W_0^+=\ker(A_0)\cap L^2(\Sigma_0,E_0^\pm)$, $V_0^+=L^2_{(0,\infty)}(A_0)$, $W_0^-=\ker(A_0)\cap L^2(\Sigma_0,E_0^\mp)$, $g_0=\pm G$ and similarly for $\Sigma_1$.
\end{example}

In order to conclude from Lemma~\ref{prop:graph} that the chirality conditions yield a Fredholm operator one would need to verify condition~\eqref{g0g1kompakt} or \eqref{g0g1klein}.
Condition~\eqref{g0g1kompakt} is not satisfied but for \eqref{g0g1klein} this is not clear.
So let us look at a special case.

\begin{example}\label{exa:statisch}
Let $h$ be a fixed Riemannian metric on the closed spin manifold $\Sigma$ and equip $M=\R\times\Sigma$ with the ``ultrastatic'' metric $-dt^2+h$.
Then \eqref{eq:DiracSurface+} simplifies to
$$
D=-\gamma(\nu)(\nabla_\nu + \rmi A)=-\gamma(\nu)\bigg(-\frac{\partial}{\partial t}+ \rmi A\bigg)
$$
where $A$ is the Dirac operator on $(\Sigma,h)$.
If we now solve the equation $Du=0$ with initial condition $u(0,\cdot)=u_0$ where $u_0$ is an eigenspinor for $A$, $Au_0=\lambda u_0$, then the solution is given by $u(t,x)=e^{\rmi \lambda t}u_0(x)$.
We choose $\Sigma_j=\{j\}\times \Sigma$.
Then $Qu_0=e^{\rmi \lambda}u_0$.

Specializing even further, we put $\Sigma=S^1$ and choose $h$ such that $\Sigma$ has length $1$.
The $1$-dimensional sphere has two spin structures.
For the so-called trivial spin structure the Dirac operator $A$ has the eigenvalues $\lambda=2\pi k$ where $k\in\Z$.
Thus $e^{\rmi \lambda}=1$ and hence $Q=\id$ where we have identified $L^2(\Sigma_0,V^R)=L^2(S^1,SS^1)=L^2(\Sigma_1,V^R)$ and $A_0=A=A_1$.
Now choose 
\begin{align*}
V_0^- &= L^2_{(-\infty,0)}(A),   &V_0^+ = L^2_{(0,\infty)}(A), \quad\quad W_0^- &= \ker(A), \quad &W_0^+&=0,\\
V_1^- &= L^2_{(-\infty,0)}(A),   &V_1^+ = L^2_{(0,\infty)}(A), \quad\quad W_1^- &= 0, \quad &W_1^+&=\ker(A).
\end{align*}
Let $g_0:L^2_{(-\infty,0)}(A) \to L^2_{(0,\infty)}(A)$ be an isomorphism of Hilbert spaces and put $g_1=g_0^{-1}$.
Then 
\[
B_0=\Gamma(g_0) = \{v+g_0v\mid v\in L^2_{(-\infty,0)}(A)\} = \{g_1w+w\mid w \in L^2_{(0,\infty)}(A)\} = \Gamma(g_1) = B_1.
\]
Now $(B_0,Q^{-1}B_1)=(B_0,B_1)=(B_0,B_0)$ is clearly not a Fredholm pair.
This shows that in Proposition~\ref{prop:graph} conditions \eqref{g0g1kompakt} and \eqref{g0g1klein} cannot be dropped.
Without these assumptions boundary conditions in graph form do not give rise to a Fredholm operator in general.
\end{example}

\begin{example}
Finally, we want to point out that one also obtains Dirac-Fredholm pairs $(B_0,B_1)$ if $B_0$ is finite dimensional and $B_1$ has finite codimension or vice versa.
According to Theorem~\ref{thm:DiracFredholm} the Dirac operator with these boundary conditions is Fredholm with index $\dim (B_0) - \mathrm{codim} (B_1)$.
\end{example}


\end{document}